\documentclass[11pt]{amsart}
\usepackage{amssymb}
\usepackage{amscd}
\usepackage{color}
\usepackage{comment}
\usepackage[dvipsnames]{xcolor}
\usepackage{url}
\usepackage{hyperref}
\usepackage{enumitem}

\allowdisplaybreaks[4]

 \newtheorem{thm}{Theorem}[section]
 \newtheorem*{theorem*}{Theorem}
 \newtheorem{cor}[thm]{Corollary}
 \newtheorem{lem}[thm]{Lemma}
 
 \newtheorem{prop}[thm]{Proposition}
 
 \theoremstyle{definition}
 
\theoremstyle{remark}
 \newtheorem{rem}[thm]{Remark}

\newtheorem{exm}[thm]{Example}
 \numberwithin{equation}{section}

\newcommand{\ol}{\overline}

\newcommand{\re}{\mathop{\rm Re}\nolimits}

{\par \hfill }

\def\CC{{\mathbb C}}
\def\DD{{\mathbb D}}

\def\LL{\mathcal L}
\def\EE{\mathcal E}
\def\NN{\mathbb N}

\def\beq{\begin{equation}}
\def\eeq{\end{equation}}

\def\spam{\mathop{\rm span}\nolimits}
\def\rang{\mathop{\rm ran}\nolimits}
\def\phi{\varphi}

\newcommand{\uu}{\textnormal{\textbf{u}} }

\newcommand{\norm}[1]{\left\|#1\right\|}
\newcommand{\ip}[2]{\left\langle #1,#2\right\rangle}

\newcommand{\sigmap}{\sigma_{\rm p}}

\newcommand{\rev}[1]{{\color{black}#1}}
\newenvironment{revised}{\color{black}}{}

\title{Frame constructions associated with operator orbits}

\author{Eva A. Gallardo-Guti\'{e}rrez}
\address{Eva A. Gallardo-Guti\'errez \newline
Departamento de An\'alisis Matem\'atico y Matem\'atica Aplicada,\newline
Facultad de Matem\'aticas,
\newline Universidad Complutense de
Madrid, \newline
Plaza de Ciencias 3, 28040 Madrid,  Spain
\newline
and Instituto de Ciencias Matem\'aticas ICMAT,
\newline Madrid,  Spain }
\email{eva.gallardo@mat.ucm.es}

\author{Jonathan R. Partington}
\address{Jonathan R. Partington, \newline
School of Mathematics, \newline
University of Leeds, \newline
Leeds LS2 9JT, United Kingdom}
\email{J.R.Partington@leeds.ac.uk}

\thanks{Both authors are partially supported by Plan Nacional  I+D grant no. PID2025-169474NB-C21, Spain. The first author is also supported by
 the Spanish Ministry of Science and Innovation, through the ``Severo Ochoa Programme for Centres of Excellence in R\&D'' CEX2023-001347.}


\date{November 2026}


\begin{document}

\begin{abstract}
This paper studies frames in Hilbert spaces generated by the orbits of finitely many vectors under a single operator, presenting new results on normal operators and multiplication operators as well as on operators composed of Jordan blocks. These results generalize existing results of Cabrelli, Molter, Paternostro and Philipp  by means of techniques which deal with weighted interpolation, weighted composition operators, and Beurling--Lax theory related to shifts of finite multiplicity.
Likewise, we discuss Carleson frames and construct counterexamples to a recent conjecture of Aldroubi, Cabrelli, Krishtal and Molter, showing that the relevant necessary condition is of density type rather than the M\"untz--Sz\'asz condition.
Moreover, by using the theory of weighted composition operators, we answer
a question of Cheshmavar on frames produced by iterative multiplication.
Finally, for orbits of infinitely many vectors the underlying model is a shift of infinite multiplicity which, by a theorem of Rota, is a universal operator; consequently, a complete description of such frames along these lines would encompass the Invariant Subspace Problem.
\end{abstract}

\maketitle

\section{Introduction}
A sequence $(x_n)_{n \ge 0}$  in an infinite dimensional, separable, complex Hilbert space $H$ is said to be a {\em frame\/} if
there are constants $C_1,C_2 > 0$ such that
\[
C_1 \|x\|^2 \le \sum_{n=0}^\infty |\langle x,x_n\rangle|^2 \le C_2 \|x\|^2
\]
for all $x \in H$.
Frames extend the notion of an orthonormal basis: while a basis provides a unique expansion, a frame allows multiple representations of the same vector. This redundancy is often advantageous, as it leads to increased numerical stability and robustness against noise, losses, or perturbations of the frame elements (see, for example, \cite{meyer}).

\smallskip

Associated with any frame there are key operators, such as the analysis operator, synthesis operator, and frame operator, which play a central role in reconstruction. In particular, the frame operator is
defined as the mapping
\[
S:x \mapsto \sum_{n=0}^\infty \langle x,x_n \rangle x_n
\]
and it is
bounded, positive, and invertible, guaranteeing that every vector in $H$ can be recovered from its frame coefficients. Frames include important special cases such as tight frames and Parseval frames, which simplify reconstruction formulas. Recall that the sequence $(x_n)_{n \ge 0}$  is said to be a {\em tight frame\/} if we can take $C_1=C_2$ and a {\em Parseval frame\/} if, in addition, both $C_1$ and $C_2$ are equal to 1; or equivalently,
the \emph{analysis operator}
\[
\begin{array}{ccl}
C_{(x_n)}: H & \longrightarrow & \ell^2 \\[4pt]
x & \longmapsto & \big(\langle x, x_n \rangle\big)_{n\geq 0}
\end{array}
\]
is an isometry. It is well-known (see \cite[Chap.~4]{meyer}) that an equivalent condition for a sequence
$(x_n)_{n \ge 0}$ to be a frame on a Hilbert space $H$  is that the \emph{synthesis operator}
\begin{eqnarray}\label{eq:bddsurj}
C_{(x_n)}^*: \ell^2 & \longrightarrow  & H, \nonumber  \\ [6pt]
(a_n)_{n\geq 0} & \longmapsto & \sum_{n=0}^\infty a_n x_n
\end{eqnarray}
should be bounded and surjective. Note that $S=C_{(x_n)}^*C_{(x_n)}$.

\smallskip

As an important class of frames, \textbf{frames of iterations}, also known as \emph{dynamical frames}, arise from the action of a bounded linear operator $T$ on the Hilbert space $H$. The key question here is finding conditions on $T$ and a collection of vectors $(u_j)_{j\in J}\subset H$ such that the family
\[
\{T^n u_j : j \in J,\; n\ge 0\}
\]
forms a frame for $H$. In \cite{CMPP20}, Cabrelli, Molter, Paternostro and Philipp found necessary and sufficient conditions for the existence
of dynamical frames (see also \cite{ACNP25}). In particular, they proved that there exists a collection $(u_j)_{j\in J}\subset H$ such that
$\{T^n u_j : j \in J,\; n \ge 0\}$ is a (dynamical) frame for $H$ if and only if $T$ is similar to a contraction $C$ whose adjoint $C^*$ is strongly stable, namely, there exists a bounded invertible operator $P$ such that
$P^{-1}TP = C$, $\|C\|\le 1$, and $(C^*)^n x \to 0$ for every $x\in H$.
Moreover, in order to generate a Parseval frame of iterations, the authors showed that the equivalence holds when $T$ is indeed a contraction whose adjoint is strongly stable.
Consequently, only contractions can produce Parseval frames of iterations, which turn out to be particularly important since their canonical dual coincides with the frame itself, being relevant in applications. We refer to the recent work \cite{ACKM25} where the authors present a unified survey of dynamical sampling and
its interplay with frame theory as well as the references therein.

\smallskip

Motivated by these results, in this manuscript we study dynamical frames for specific classes of operators. In particular, in Section \ref{sec:2} we are concerned with normal operators with a Riesz basis of eigenvectors, and we relate them to classical and more recent results in interpolation theory, namely, to \emph{weighted interpolation}. In this regard, we remark that the case of diagonal operators acting on $\ell^2$ when a single vector is iterated was explicitly described in \cite{CMPP20}, connecting the problem with interpolation sequences in the Hardy space on the open unit disc $H^2(\DD)$. In particular, for diagonal operators of multiplicity one, the following theorem is proved in \cite{ACMT17}:

\smallskip

\begin{thm}[\cite{ACMT17}]\label{thm:diagonal}
Let $D = \sum_{j\in J} \lambda_j P_j$ be a diagonal operator where $J$ is a countably infinite index set, $(\lambda_j)_{j\in J}\subset \mathbb{C}$ is a
bounded sequence of scalars, and $(P_j)_{j\in J}$ is a sequence of orthogonal
projections satisfying $P_j P_k = 0$ for $j \neq k$ and $\sum_{j\in J} P_j = \mathrm{Id}$. Assume further $\lambda_j \neq \lambda_k$ for $j \neq k$ and
each $P_j$ is rank one and let $\uu =(u_k)_{k\geq 0} \in \ell^2$. Then the family
\[
\{D^{\ell} \uu : \ell = 0,1,\ldots\}
\]
is a frame if and only if the following conditions hold:
\begin{enumerate}
\item $|\lambda_k| < 1$ for all $k$;
\item $|\lambda_k| \to 1$;
\item $(\lambda_k)_{k\geq 0}$ satisfies Carleson’s condition
\[
\inf_n \prod_{k \neq n}
\frac{|\lambda_n - \lambda_k|}{|1 - \overline{\lambda_n}\lambda_k|}
\ge \delta,
\]
for some $\delta > 0$;
\item $u_k= m_k \sqrt{1 - |\lambda_k|^2}$ for some sequence $(m_k)_{k\geq 0}$ satisfying
\[
0 < C_1 \le |m_k| \le C_2 < \infty.
\]
\end{enumerate}
\end{thm}

The condition that the eigenvalues have modulus strictly less than one ensures that the operator is contractive and strongly stable. In \cite[Theorem~2.3]{CMPP20} the authors have generalized Theorem~ \ref{thm:diagonal} to arbitrary finite collection $\{\uu_i\}_{i\in I}$ by a rather non-trivial treatment, as they explicitly point out.

\smallskip

In Section~\ref{sec:3} we deal with \emph{Carleson frames} or \emph{Alcamota frames}, namely frames satisfying the stated conditions (1)--(4).  As pointed out in \cite{ACKM25}, such frames have \emph{excessive redundancy}.
Indeed, the subfamily obtained by selecting each $N$th element from a
Carleson frame remains a frame under a mild additional condition on the spectrum of the diagonal operator $D$ regardless of the choice of $N \in \mathbb{N}$ (see \cite{ChristensenHasannasabPhilippStoeva2024} and \cite{KrishtalMiller2026}). {If $b\in\ell^{2}$, and the orbit $\{D^{\ell}b:\ell=0,1,\dots\}$ forms a Carleson frame for $\ell^{2}$, we first note that not every set $\Lambda=\{\lambda_k\}_{k\geq 0}\subset(0,\infty)$ satisfying the Müntz--Szász divergence condition
\begin{equation}\label{MS condition}
\sum_{\lambda_k\in\Lambda}\frac{1}{\lambda_k}=\infty,
\end{equation}
defines an orbit $\{D^{\lambda}b:\lambda\in\Lambda\}$ that is again a frame for $\ell^{2}$, and we prove that both  \eqref{MS condition} and a density-type condition (\emph{a Carleson-type condition}) are necessary to be a frame in $\ell^{2}$. \rev{In fact, the M\"untz--Sz\'asz condition turns out to be a consequence of the Carleson-type condition.} The proof proceeds by passing to a spectral/functional model in which $D$ is represented as multiplication by $t\in(0,1)$ on a Hilbert space $L^2(\nu)$.

\smallskip

In Section~\ref{sec:4} we consider multiplication operators on the Hardy space $H^2(\DD)$ and use weighted composition operators to describe the associated frames.
Finally, Section~\ref{sec:6} is motivated by a result in \cite{CHP20} expressing operator frames in terms of model spaces. We use this to analyse several \rev{particular classes of frames, including operators composed of Jordan blocks}, and then discuss the extension to the multivariable case due to Cabrelli, Molter and Su\'arez~\cite{CMS23}.

\smallskip

\rev{We close the paper with a remark on orbits of infinitely many vectors: in that case the shift of finite multiplicity is replaced by a shift of infinite multiplicity, whose adjoint is a universal operator in the sense of Rota \cite{Rota}, so that the model-space approach cannot be expected to yield a classification comparable to the finite multiplicity case without, in particular, settling the invariant subspace problem.}

\rev{
\subsection*{Acknowledgements}
In the stage of reviewing this manuscript, the authors used the AI assistant Claude (Anthropic) to draft part of the argument in the example provided in  Proposition~\ref{prop:jordan}.}

\section {Normal operators: the weighted interpolation approach} \label{sec:2}

Recall that a \emph{normal operator} on a Hilbert space $H$ is a bounded linear operator $T$ that commutes with its adjoint, that is,
$T T^* = T^* T.$ Equivalently, $T$ is normal if
$\|Tx\| = \|T^*x\| \quad \text{for all } x \in H$. Normal operators include important classes such as self-adjoint and unitary operators,
and they admit a spectral theorem that represents them as multiplication operators on suitable $L^2$-spaces.

\smallskip

Assume $T$ is a normal operator. We start  by pointing out that the authors in \cite{ACCMP17} proved  that there does not exist a Riesz basis of iterations
for a normal operator, that is, for any collection of vectors $G \subset H$, the family of iterates $\{T^n u : u \in G,\; n \ge 0\}$ is never a Riesz basis of $H$. Recall that a \emph{Riesz basis} in $H$ is a sequence of vectors
$(f_k)_{k\geq 0}\subset H$ that is the image of an orthonormal basis under a bounded,
invertible linear operator. Equivalently, $(f_k)_{k\geq 0}$ is a Riesz basis if it is complete
in $H$ and there exist constants $0<A\le B<\infty$ such that
\[
A \sum_k |c_k|^2 \;\le\;
\Big\|\sum_k c_k f_k\Big\|^2
\;\le\;
B \sum_k |c_k|^2,
\]
for all finite scalar sequences $(c_k)_{k\geq 0}$. Unlike general frames, a Riesz basis yields unique expansions of every vector in $H$,
while still allowing more flexibility than an orthonormal basis. In particular, Riesz bases are precisely the non-redundant frames.

\medskip

In this regard, the goal is to investigate the existence of redundant frames for normal operators. Obviously, since unitary operators are
never strongly stable, they cannot generate frames of iterations.
\smallskip

Our aim at this stage is making use of weighted interpolation studied by McPhail \cite{mcphail} and generalized in  \cite{JPP07} to provide an extension of \cite[Theorem~2.3]{CMPP20}. For the sake of clarity, we will first deal with the case when the collection $\{\uu_i\}_{i\in I}$ is reduced to just one vector $\uu$.

\smallskip

\subsection*{Interpolation in the scalar case}

Assume $T$ is a diagonalizable normal operator in $\ell^2$  with eigenvalues $(\lambda_k)_{k \ge 0}$ and orthonormal eigenvectors $(\phi_k)_{k \ge 0}$ (in fact, the results we prove will also apply if the eigenvectors form merely a Riesz basis). Let us write $\uu=\sum_{k=0}^\infty u_k \phi_k \in \ell^2$.

\smallskip

Conditions under which $(T^n \uu)_{n \ge 0}$ forms a frame were established in
\cite{ACMT17} and \cite{CMPP20} as mentioned. In order to revisit these results using interpolation
techniques to extend them further, we conveniently work in the Hardy space $H^2(\DD)$. Recall that $H^2(\DD)$ consists of all analytic functions $f$ on the open unit disk $\mathbb{D}$
\[
f(z) = \sum_{n=0}^\infty a_n z^n, \qquad z \in \mathbb{D},
\]
where the sequence of Taylor coefficients $(a_n)_{n\ge 0} \in \ell^2$ and $\|f\|_{H^2}= \|(a_n)_{n\geq 0}\|_{\ell^2}$.

\smallskip

\noindent Hence, mapping the orthonormal basis elements $\phi_n$ to $z^n$, we find that for $f(z)=\sum_{n=0}^\infty a_n z^n\in H^2$,
\[
\sum_{n=0}^\infty  a_n T^n \uu = \sum_{n=0}^\infty a_n \sum_{k=0}^\infty   \lambda_k^n u_k z^k
= \sum_{k=0}^\infty f(\lambda_k)u_k z^k.
\]
Consequently, the boundedness and surjectivity of the synthesis operator associated to $(T^n \uu)_{n \ge 0}$, namely, $C_{(T^n \uu)}^*$, is equivalent to the solution of the weighted interpolation problem
\begin{equation}\label{weighted interpolation}
f(\lambda_k)u_k= c_k \qquad (k\geq 0)
\end{equation}
for all sequences $(c_k)_{k\geq 0}$ in $\ell^2$.

\smallskip

The classical interpolation problem asks if there exist a sequence of points $(z_n)_{n\geq 0}$ in $\mathbb{D}$ having the property that, given an arbitrary bounded sequence of complex numbers $(a_n)_{n\geq 0}$, there exists a bounded analytic function $f$ on $\mathbb{D}$ such that
\begin{equation*}
f(z_n) = a_n \qquad (n\geq 0)
\end{equation*}
The existence of such a sequence, called an \emph{interpolating sequence}, was
established independently by L.~Carleson~\cite{Carleson}, W.~Hayman~\cite{Hayman},
and D.~J.~Newman~\cite{Newman}. In fact, Carleson showed that a sequence
$(z_n)_{n\geq 0}$ in $\mathbb{D}$ is an interpolating sequence if and only if there exists
$\delta > 0$ such that
\begin{equation}\label{Carleson}
\prod_{n \neq k}
\left|
\frac{z_k - z_n}{1 - \overline{z}_n z_k}
\right|
\ge \delta
\qquad (k = 0,1,2,\ldots).
\end{equation}
Clearly, such sequences consist of distinct points. For a discussion of independent and related results, see H.~S.~Shapiro and
A.~L.~Shields~\cite{ShapiroShields} and Garnett's monograph \cite{Garnett}.

\smallskip

In 1990, McPhail \cite{mcphail} studied a weighted interpolation problem which, in particular, extended Shapiro and Shields's results. More precisely,  given a sequence $(z_n)_{n\geq 0}$ in $\mathbb{D}$, let $\Lambda = \Lambda_{(z_n)}$ denote the linear map sending an analytic function
$f$ on $\mathbb{D}$ to the sequence of its values on the $z_n$, namely
$(f(z_n))_{n\geq 0}$. In this setting, Carleson’s result says that
$\Lambda (H^\infty) \supset \ell^\infty $, or equivalently, $\Lambda (H^\infty) = \ell^\infty$
if and only if condition~ \eqref{Carleson} holds. Here $H^\infty$ denotes the Hardy space of bounded analytic functions on $\mathbb{D}$ endowed with the sup norm
$\|f\|_{H^\infty}= \sup_{z \in \mathbb{D}} |f(z)|$.

\medskip

Given a sequence $w =(w_n)_{n\geq 0}$ of positive real numbers, let
$\ell^p(w)$ denote the space of sequences $a = (a_n)_{n\geq 0}$ satisfying
\begin{equation}
\|a\|_{\ell^p(w)} =
\begin{cases}
\displaystyle
\left( \sum_{n=0}^\infty |a_n w_n|^p \right)^{1/p} < \infty,
& 1 \le p < \infty, \\
\noalign{\smallskip} \\[8pt]
\displaystyle
\sup_{n \ge 0} |a_n w_n| < \infty,
& p = \infty.
\end{cases}
\end{equation}

The McPhail theorem characterizes the sequences $(z_n)_{n\geq 0}$ in $\mathbb{D}$ such that
$\Lambda_{(z_n)} (H^p(\nu)) = \ell^p(w)$ for some weighted Hardy spaces $H^p(\nu)$ (see \cite{mcphail} for more details).
As a particular instance, the Shapiro and Shields~\cite{ShapiroShields} theorem follows if $\nu = 1$ and $w_n = (1 - |z_n|)^{1/p}$ for  $(n \ge 0)$, that is,
\[
\Lambda_{(z_n)}(H^p) \supset \ell^p(w)
\]
(in fact, equivalently $\Lambda_{(z_n)}(H^p) = \ell^p(w)$) if and only if condition~ \eqref{Carleson} holds.

\smallskip

We now quote a simplified version of the theorem of McPhail \cite{mcphail} for our purposes:


\begin{thm}[\cite{mcphail}]\label{thm:mcphail}
Let $w =(w_n)_{n\geq 0}$ be a sequence of positive real numbers and $(z_n)_{n\geq 0}\subset \mathbb{D}$. Then
$\Lambda_{(z_n)}(H^2) = \ell^2(w)$ if and only if   condition~\eqref{Carleson} holds and there exist constants
$C_1, C_2 > 0$ such that
\[
C_1 \le w_n/\sqrt{1-|z_n|} \le C_2
\]
for all $n$.
\end{thm}

With Theorem~\ref{thm:mcphail} at hand, the weighted interpolation problem \eqref{weighted interpolation} is solvable and hence, Theorem \ref{thm:diagonal} follows readily.

\begin{cor}\label{cor:scalar}
Assume $T$ is a diagonalizable normal operator in $\ell^2$  with eigenvalues $(\lambda_k)_{k \ge 0}$ associated to a (Riesz) basis of eigenvectors $(\phi_k)_{k \ge 0}$ and let $\uu=\sum_{k=0}^\infty u_k \phi_k \in \ell^2$. The sequence $(T^n \uu)_{n\geq 0}$ forms a frame in $\ell^2$ if and only if the sequence $(\lambda_k)_{k \ge 0}$ is an interpolating sequence in $\DD$ and there exist constants $C_1, C_2 > 0$ such that
\[
C_1 \le |u_k|/\sqrt{1-|\lambda_k|}
\le C_2
\]
for all $k$.
\end{cor}

\subsection*{The multivariable case} Having discussed the scalar case, we deal with the multivariable case. Without loss of generality, we may assume that $T$ is a diagonalizable normal operator in $H^2(\DD)$  with eigenvalues $(\lambda_k)_{k \ge 0}$ associated to a (Riesz) basis of eigenvectors. Let $\{\uu_1,\ldots,\uu_m\}$ be a set in $H^2(\DD)$.  Our aim is studying when
$$\{T^n \uu_k: n \ge 0, k=1,\ldots,m\}$$
is a frame in $H^2(\DD)$.

\smallskip

For such a goal, we will make use of a more general form of McPhail's theorem, given in \cite{JPP07}, pages 537--538 (see also \cite{JZ01}). The aim is solving a vector-valued weighted interpolation problem a little more involved than \eqref{weighted interpolation}, namely,
\begin{equation}\label{weighted interpolation2}
G_k f(\lambda_k)= c_k \qquad (k\geq 0)
\end{equation}
for all sequences $(c_k)_{k\geq 0}$ in $\ell^2$, where the $G_k$ are rank-1 operators.

\smallskip

In order to state the vector-valued weighted interpolation theorem required, one change has to be made, as these works \cite{JZ01,JPP07} take place in the setting of the right-half plane
$$\CC_+=\{ z \in \mathbb{C} : \Re z > 0 \}.$$
Recall that the Hardy space $H^2(\mathbb{C}_+)$ of the right half-plane consists of all functions
$f$ that are analytic in $\mathbb{C}_+$ and satisfy
\[
\|f\|_{H^2(\mathbb{C}_+)}^2=
\sup_{x>0}
\int_{-\infty}^{\infty} |f(x+iy)|^2 \, dy
< \infty.
\]
It is worth remarking that the norms on $H^2(\mathbb{C}_+)$ and $L^2(0,\infty)$ are equivalent via the Laplace
transform. More precisely, the Laplace transform
\[
\mathcal{L}f(s)= \int_0^\infty e^{-st} f(t)\,dt,
\qquad s \in \mathbb{C}_+,
\]
defines (up to a constant) an isometric isomorphism from $L^2(0,\infty)$ onto $H^2(\mathbb{C}_+)$, that is,
$\|\mathcal{L}f\|_{H^2(\mathbb{C}_+)}\approx\|f\|_{L^2(0,\infty)}$ (see \cite{LOLS}, for instance).

\smallskip

Let $M$ denote the involution
$$M(s)=\frac{1-s}{1+s}, \qquad (s\in \CC_+)$$
between $\CC_+$ and $\DD$. We have that $f\in H^2(\DD)$ if and only if $Vf \in H^2(\CC_+)$ with an equivalence of norms, where
$$Vf(s)= \frac{1}{1+s} f(M(s)), \qquad (s \in \CC_+)$$
(see, e.g. \cite[Chap. 8]{hoffman}). We summarize the change of coordinates as follows.

\begin{prop}\label{prop:change}
Let $(z_k)_{k\geq 0}$ be a sequence in $\DD$ and $(c_k)_{k\geq 0}$ a sequence of complex numbers. Then given a sequence $(G_k)$ of linear mappings $G_k: \CC^N \to \CC$, $N\geq 1$, the function $f \in \rev{H^2(\DD;\CC^N)}$ satisfies the interpolation conditions
$$G_k f(z_k)=c_k$$
for each $k$ if and only if the function $F=Vf$ in $\rev{H^2(\CC_+;\CC^N)}$ \rev{(with $V$ acting componentwise)} satisfies the interpolation condition
$$\rev{(1+s_k)}\,G_k (Vf)(s_k)=c_k$$
where $s_k=M(z_k)$ for each $k$.
\end{prop}

Let $H^2(\mathbb{D}; \mathbb{C}^N)$ denote the space of $\mathbb{C}^N$-valued functions $f = (f_1,\ldots,f_N)$ that are analytic on $\mathbb{D}$
and $\|f\|_{H^2(\mathbb{D}; \mathbb{C}^N)}^2:= \sum_{k=1}^N \|f_k\|_{H^2(\mathbb{D})}^2$, or equivalently,
\[
\|f\|_{H^2(\mathbb{D}; \mathbb{C}^N)}^2= \sup_{0<r<1} \frac{1}{2\pi}
\int_{0}^{2\pi} \|f(re^{i\theta})\|_{\mathbb{C}^N}^2 \, d\theta
< \infty,
\]
where $\|\cdot\|_{\mathbb{C}^N}$ is the Euclidean norm on $\mathbb{C}^N$. As a consequence, the following conditions are equivalent:

\begin{itemize}
  \item $(G_k)_{k\geq 0}$ and $(z_k)_{k\geq 0}$ have the property that for all $(c_k)_{k\geq 0} \in \ell^2$
  we can find $f \in H^2(\DD; \CC^N)$ such that $$G_k f(z_k)=c_k$$ for all $k$;
  \item $(G_k)_{k\geq 0}$ and $(s_k)_{k\geq 0}=(M(z_k))_{k\geq 0}$ have the property that for all $(c_k)_{k\geq 0} \in \ell^2$
  we can find $F \in H^2(\CC_+; \CC^N)$ such that $$(1+s_k)G_k F(s_k)=c_k$$ for all $k$.
  \end{itemize}

Now, Theorem 2.24 of \cite{JPP07}, combined with Theorem 2 of \cite{JZ01}, gives the following result for scalar interpolation in $H^2(\CC_+;\CC^N)$.
We write $\EE(F)=(G_k F(s_k))$ for $F \in H^2(\CC_+;\CC^N)$, where each $G_k$
is a linear mapping from $\CC^N$ to $\CC$, and thus has the form $G_k(x)=\langle x, g_k \rangle$ for some $g_k \in \CC^N$.

\begin{thm}
 The following are equivalent:
 \begin{itemize}
 \item $\EE(H^2(\CC_+;\CC^N))= \ell^2$;
 \item $(s_k)$ is the union of at most $N$ Carleson sequences, the
 quantities\\ $\|g_k\|^2/\re s_k$ are bounded above and below, and there exists a constant
 $r>0$ such that
 \[
 \inf_{m \in \NN} \min_{s_n \in \Lambda_m(r)} \angle (g_n e^{-s_n t}, \spam_{s_j\in\Lambda_m(r), j \ne n}
 \{ g_j e^{-s_j t}\} ) > 0,
 \]
 where $\Lambda_m(r) = \{ s_n : \left| \dfrac{s_n-s_m}{s_n+\overline{s_m}} \right| < r \}$
 and the angle is calculated in $L^2(0,\infty; \CC^N)$.
 \item The set $\{ g_k e^{-s_k t}: k \in \NN\}$ is a Riesz basic sequence
 in $L^2(0,\infty; \CC^N)$.
 \end{itemize}
 \end{thm}

 Note that the angle may  be expressed in terms of norms of operators, as in \cite[Lem.~2.15]{JPP07}.

On translating this back to the disc, noting that
$$\frac{1}{1+s_k}=\frac{1+z_k}{2}\text{ and } \re s_k= (1-|z_k|^2)/|1+z_k|^2,$$
and having in mind that the norms in $H^2(\CC_+)$ and  $L^2(0,\infty)$  are equivalent by the Laplace
transform, together with the properties of $V$ above, we obtain the following,
on writing $ \widetilde \EE(f)=(G_k f(z_k))$ for $f \in \rev{H^2(\DD;\CC^N)}$ and
$G_k(x)=\langle x, g_k\rangle$ again.
\rev{Note that, by Proposition~\ref{prop:change}, the relevant functionals in the half-plane are $(1+s_k)G_k$, which correspond to the vectors $(1+\overline{s_k})g_k$, and
\[
\frac{|1+s_k|^2\,\|g_k\|^2}{\re s_k}=\frac{4\,\|g_k\|^2}{1-|z_k|^2}.
\]
Moreover, $M$ preserves the pseudo-hyperbolic distance, so that Carleson sequences in $\CC_+$ correspond to Carleson sequences in $\DD$, and, up to non-zero scalar factors, the functions $e^{-\overline{s_k} t}$ correspond to the reproducing kernels $\mathbf{K}_{z_k}(z)=1/(1-\overline{z_k}z)$ of $H^2(\DD)$.}

 \begin{thm}\label{thm:interpold}
 The following are equivalent:
 \begin{itemize}
 \item $\widetilde \EE(\rev{H^2(\DD;\CC^N)})= \ell^2$;
 \item $(z_k)$ is the union of at most $N$ Carleson sequences, the
 quantities\\ $\rev{\|g_k\|^2/(1-|z_k|^2)}$ are bounded above and below, and there exists a constant
 $r>0$ such that
 \[
 \inf_{m \in \NN} \min_{z_n \in \Lambda_m(r)} \angle \rev{\Big(\frac{g_n}{1-\overline{z_n} z}, \spam_{z_j\in\Lambda_m(r), j \ne n}
 \Big\{ \frac{g_j}{1-\overline{z_j} z}\Big\} \Big)} > 0,
 \]
 where $\Lambda_m(r) = \{ z_n : \left| \dfrac{z_n-z_m}{\rev{1 - \overline{z_m}z_n}} \right| < r \}$
 and the angle is calculated in $H^2(\DD; \CC^N)$.
 \item The set $\rev{\{g_k/(1-\overline{z_k} z): k \geq 0\}}$ is a Riesz basic sequence
 in $H^2(\DD; \CC^N)$.
 \end{itemize}
 \end{thm}

\rev{The equivalence of the first and the last condition can also be seen directly: since
$$\langle \widetilde\EE f, e_k\rangle_{\ell^2}=\langle f(z_k),g_k\rangle_{\CC^N}=\langle f, g_k\mathbf{K}_{z_k}\rangle_{H^2(\DD;\CC^N)},$$
we have $\widetilde\EE^* e_k=g_k\mathbf{K}_{z_k}$, and $\widetilde\EE$ is bounded and onto $\ell^2$ if and only if $\widetilde\EE^*$ is bounded and bounded below.}

\smallskip

For $N=1$, where the $g_k$ are scalars, this reduces to saying that $(z_k)$ is a Carleson sequence
and the $|g_k|/\sqrt{1-|z_k|}$ are bounded above and below (the condition on angle is vacuous \rev{in this case: a Carleson sequence is uniformly separated in the pseudo-hyperbolic metric, so taking $r$ smaller than the separation constant we get $\Lambda_m(r)=\{z_m\}$ for every $m$}). This agrees with McPhail's
Theorem \ref{thm:mcphail}.

\begin{revised}
\subsection*{Frames generated by finitely many vectors}
We can now state the frame version of Theorem~\ref{thm:interpold}, which extends Corollary~\ref{cor:scalar} to finitely many vectors.

\begin{cor}\label{cor:multivector}
Let $T$ be a diagonalizable normal operator on a separable Hilbert space $H$ with eigenvalues $(\lambda_k)_{k\geq 0}$ \textup{(}repeated according to multiplicity\textup{)} associated to an orthonormal \textup{(}or Riesz\textup{)} basis of eigenvectors $(\phi_k)_{k\ge 0}$, and let $\uu_1,\ldots,\uu_m\in H$, say $\uu_j=\sum_{k\geq 0}u_{j,k}\phi_k$. Write
\[
g_k=\big(\overline{u_{1,k}},\ldots,\overline{u_{m,k}}\big)\in\CC^m \qquad (k\geq 0).
\]
Then $\{T^n\uu_j: n\geq 0,\ j=1,\ldots,m\}$ is a frame for $H$ if and only if the equivalent conditions of Theorem~\ref{thm:interpold} hold with $N=m$ and $z_k=\lambda_k$. In particular, if it is a frame then $(\lambda_k)$ is the union of at most $m$ Carleson sequences \textup{(}so every eigenvalue has multiplicity at most $m$\textup{)} and
\[
\sum_{j=1}^m |u_{j,k}|^2 \asymp 1-|\lambda_k|^2 \qquad (k\geq 0).
\]
\end{cor}

\begin{proof}
For finitely supported sequences $(a_{j,n})$ put $f_j(z)=\sum_{n} a_{j,n}z^n$ and $f=(f_1,\ldots,f_m)\in H^2(\DD;\CC^m)$. Then
\[
\sum_{j=1}^m\sum_{n\geq 0} a_{j,n}T^n\uu_j=\sum_{k\geq 0}\Big(\sum_{j=1}^m f_j(\lambda_k)u_{j,k}\Big)\phi_k
=\sum_{k\geq 0}\langle f(\lambda_k),g_k\rangle_{\CC^m}\,\phi_k .
\]
Hence, up to the unitary identifications $\ell^2(\NN_0\times\{1,\ldots,m\})\cong H^2(\DD;\CC^m)$ and $H\cong\ell^2$ (an isomorphism if $(\phi_k)$ is merely a Riesz basis), the synthesis operator of $\{T^n\uu_j\}$ is the operator $\widetilde\EE$ with $z_k=\lambda_k$. By the Closed Graph Theorem it is bounded and surjective if and only if $\widetilde\EE(H^2(\DD;\CC^m))=\ell^2$, and the result follows from Theorem~\ref{thm:interpold}.
\end{proof}
\end{revised}

\subsection*{An explicit expression: the bivariable case} Let us write out in detail the interpolation question linked with
the condition for $\{T^n u,T^n v: n \ge 0\}$ to be  a frame.

The mapping
\[
\begin{array}{rcl}
\ell^2(\mathbb{C}^2) & \longrightarrow & H \\[4pt]
(a_n,b_n)_n & \longmapsto & \displaystyle \sum_{n} a_n T^n u + b_n T^n v
\end{array}
\]
should be bounded and surjective. Writing $f(z)=\sum_{k=0}^\infty a_k z^k$ and $g(z)=\sum_{k=0}^\infty b_k z^k$, with $f,g \in H^2$, we
have
\[
\sum_{n=0}^\infty  a_n T^n u+b_n T^n v = \sum_{k=0}^\infty f(\lambda_k)u_k z^k + g(\lambda_k) v_k z^k,
\]
so the interpolation problem is
$$f(\lambda_k)u_k+  g(\lambda_k)v_k=c_k,$$
for every $k\geq 0$, which fits into the framework of Theorem~\ref{thm:interpold} \rev{with $N=2$ and $g_k=(\overline{u_k},\overline{v_k})$ (see Corollary~\ref{cor:multivector}).}

\section{Carleson frames and the M\"untz--Sz\'asz condition}\label{sec:3}

For an operator $D$ we write $\sigma(D)$ to denote its spectrum and $\sigmap(D)$ to denote
its point spectrum (the set of eigenvalues).

As pointed out in the Introduction, Carleson frames (or Alcamota frames), namely, frames satisfying conditions (1)--(4) in Theorem \ref{thm:diagonal}, exhibit excessive redundancy. Motivated by the recent work in \cite{ChristensenHasannasabPhilippStoeva2024} and \cite{KrishtalMiller2026}, the authors in \cite{ACKM25} conjectured that, if $b\in\ell^{2}$ and $\sigmap(D)\subset (0,1)$, then the orbit $\{D^{\ell}b:\ell=0,1,\dots\}$ forms a Carleson frame for $\ell^{2}$ and  $\Lambda \subset (0,\infty)$ is a set satisfying the M\"untz--Sz\'asz condition
\begin{equation}\label{MS condition-2}
\sum_{\lambda \in \Lambda} \frac{1}{\lambda} = \infty,
\end{equation}
then the system $\{D^{\lambda} b : \lambda \in \Lambda\}$ is a frame for $\ell^2$.}

\begin{rem}{\rm
It is easy to see that \eqref{MS condition-2} by itself is not sufficient. Indeed, let $D$ be a diagonal operator such that $D\varphi_k=\mu_k \varphi_k$ for $\{\varphi_k\}_{k\geq 0}$ an
orthonormal basis of an infinite dimensional separable complex Hilbert space $H$. Suppose that $\{D^n u:\; n\geq 0\}$ is a frame for some $u \in H$, so that $\{\mu_k\}_{k\geq 0}$ is a Carleson sequence; we shall also assume that $0< \mu_k < 1$ for each $k$.

Now take a sequence $\Lambda=\{\lambda_n\}_{n\geq 0}$ such that $ \lambda_n>0$ for each $n$ and
$\lambda_n \to 1$. We will take the $\lambda_n$ distinct, and we see that
$\sum_{n=0}^\infty 1/\lambda_n = \infty.$ (Later we shall see the stronger result that even taking {\em subsequences of $\NN$}
satisfying the M\"untz--Sz\'asz condition is not enough to guarantee a frame.)

We claim that for every vector $v \in H$, say, $v = \sum_{k=0}^\infty v_k \varphi_k$, the
sequence $\{D^{\lambda_n} v:\; n\geq 0\}$ is not a frame. Clearly we may assume that $v \ne 0$, so let $k$ be chosen such that $v_k \ne 0$.
Now $\langle \varphi_k, D^{\lambda_n} v \rangle_H = \ol {v_k} \mu_k^{\lambda_n}$ and
so
\[
\sum_{n=0}^\infty |\langle \varphi_k, D^{\lambda_n} v \rangle_H|^2 =\sum_{n=0}^\infty |v_k|^2 \mu_k^{2\lambda_n} = \infty,
\]
since $\mu_k^{\lambda_n} \to \mu_k\neq 0$. Thus even the upper frame (Bessel) bound fails.
}
\end{rem}

Our next result establishes that $\{D^{\lambda}b:\lambda\in\Lambda\}$ is a frame for $\ell^2$ if and only if $\{t^{\lambda}\}_{\lambda\in\Lambda}$ is a frame in an associated $L^2(\nu)$ space. Recall that if  \(X\) is a measurable space and \(x \in X\),  the \emph{delta measure} (or \emph{Dirac measure}) at \(x\), denoted by \(\delta_x\), is the measure defined by
\[
\delta_x(A) =
\begin{cases}
1, & \text{if } x \in A, \\
0, & \text{if } x \notin A,
\end{cases}
\]
for every measurable set \(A \subseteq X\). Equivalently, $\delta_x(A) = \mathbf{1}_A(x),$ where \(\mathbf{1}_A\) is the indicator function of the set \(A\).

\begin{thm}\label{prop-ms}
Let $D$ be a diagonal operator on $\ell^2$ with eigenvalues $\mu_k\in(0,1)$, $k\geq 0$, and $\{\phi_k\}_{k\geq 0}$ an orthonormal basis such that $D\phi_k=\mu_k\phi_k$ for every $k$.
Let $b\in\ell^2$ and assume that $\{D^{\ell} b:\ell=0,1,2,\dots\}$ is a Carleson frame for $\ell^2$. If $\Lambda=\{\lambda_n\}_{n\geq 0}$ is any sequence of non-negative real numbers, then $\{D^\lambda b:\lambda\in\Lambda\}$
is a frame for $\ell^2$ if and only if $\{t^{\lambda}\}_{\lambda\in\Lambda}$ is a frame in $L^2(\nu)$ where $\nu$ is the discrete measure
\[
\nu= \sum_{k\ge 0} (1-\mu_k^2)\,\delta_{\mu_k}
\quad\text{on } [0,1).
\]
\end{thm}

The proof is based on a spectral model in which $D$ is represented as multiplication by $t\in(0,1)$ on a Hilbert space determined by $(D,b)$.
In this model, the vectors $D^{\lambda}b$ correspond precisely to the functions $t^{\lambda}$.

\begin{proof}
Let us write $b=\sum_k b_k\phi_k$ and note that, since $\{D^\ell b\}_{\ell\ge 0}$ is a Carleson frame and $\mu_k\in(0,1)$ for all $k$,
the frame characterization for diagonal normal operators yields:
\begin{itemize}
\item $(\mu_k)$ is an interpolating sequence in $\mathbb D$,
\item there exist constants $C_1,C_2>0$ such that
\begin{equation}\label{uk}
C_1 \le \frac{|b_k|}{\sqrt{1-\mu_k^2}} \le C_2
\quad\text{for all } k.
\end{equation}
\end{itemize}
For $x=\sum_{k\geq 0} x_k\phi_k\in\ell^2$ and $\lambda>0$, note that
\[
\ip{x}{D^\lambda b}
= \sum_{k\ge 0} x_k\overline{b_k}\,\mu_k^\lambda.
\]
Define $J:\ell^2\to L^2(\nu)$ by
\[
(Jx)(\mu_k) := \frac{x_k\overline{b_k}}{1-\mu_k^2}.
\]
Using \eqref{uk}, $J$ is bounded and bounded below, and
\[
\norm{Jx}_{L^2(\nu)}^2
= \sum_{k\geq 0} \frac{|b_k|^2}{1-\mu_k^2}|x_k|^2
\asymp \norm{x}_{\ell^2}^2.
\]
Furthermore, $J$ is onto: given $f\in L^2(\nu)$, the sequence $x_k=(1-\mu_k^2)f(\mu_k)/\overline{b_k}$ satisfies $\sum_k|x_k|^2=\sum_k (1-\mu_k^2)\frac{1-\mu_k^2}{|b_k|^2}|f(\mu_k)|^2\le C_1^{-2}\|f\|^2_{L^2(\nu)}$ by \eqref{uk}, and $Jx=f$. Hence $J$ is an isomorphism of $\ell^2$ onto $L^2(\nu)$.
Moreover,
\[
\ip{x}{D^\lambda b}
= \sum_{k\geq 0} (1-\mu_k^2)(Jx)(\mu_k)\,\mu_k^\lambda
= \ip{Jx}{t^\lambda}_{L^2(\nu)}.
\]
Hence, $\{D^\lambda b\}_{\lambda\in\Lambda}$ is a frame for $\ell^2$ if and only if $\{t^\lambda\}_{\lambda\in\Lambda}$ is a frame for $L^2(\nu)$, as we claimed.
\end{proof}

With Theorem \ref{prop-ms} at hand, in order to approach the conjecture in \cite{ACKM25}, our task is reduced to studying when $\{t^\lambda\}_{\lambda\in\Lambda}$ is a frame for $L^2(\nu)$ assuming that $\{t^\ell\}_{\ell\in\mathbb{N}}$ is. For such a goal, let us write for the moment $w_k = 1 - \mu_k^2$ for $k\geq 0$. Note that the frame inequality would require
\[
A \|f\|_{L^2(\nu)}^2
\le
\sum_{\lambda \in \Lambda}
|\langle f, t^\lambda \rangle|^2.
\]
Testing on the normalized atom
\begin{equation}\label{ek}
e_k =
\frac{\mathbf{1}_{\{\mu_k\}}}{\sqrt{w_k}},
\end{equation}
we get
$
\|e_k\|_{L^2(\nu)} = 1
$
and
\[
\langle e_k, t^\lambda \rangle
=
\sqrt{w_k}\,\mu_k^\lambda.
\]
Therefore a necessary condition for the lower frame bound is
\[
\inf_{k\geq 0}
w_k
\sum_{\lambda \in \Lambda}
\mu_k^{2\lambda}
> 0.
\]
That is,
\[
\inf_{k\geq 0}
(1 - \mu_k^2)
\sum_{\lambda \in \Lambda}
\mu_k^{2\lambda}
> 0.
\]
In general, $\{t^\lambda\}_{\lambda \in \Lambda}$ is a frame for $L^2(\nu)$ exactly when there are constants $A,B > 0$ such that, for every finitely supported scalar sequence $(a_k)$,
\[
A \sum_{k\geq 0} |a_k|^2 w_k
\le
\sum_{\lambda \in \Lambda}
\left|
\sum_{k\geq 0} a_k w_k \mu_k^\lambda
\right|^2
\le
B \sum_{k\geq 0} |a_k|^2 w_k.
\]
After normalizing \(c_k = a_k \sqrt{w_k}\),
\[
A \sum_{k\geq 0} |c_k|^2
\le
\sum_{\lambda \in \Lambda}
\left|
\sum_{k\geq 0} c_k \sqrt{1 - \mu_k^2}\, \mu_k^\lambda
\right|^2
\le
B \sum_{k\geq 0} |c_k|^2.
\]
So, a necessary pointwise consequence is
\begin{equation}\label{pointwise condition}
0 < \inf_{k\geq 0} (1 - \mu_k^2)\sum_{\lambda \in \Lambda} \mu_k^{2\lambda}
\le
\sup_{k\geq 0} (1 - \mu_k^2)\sum_{\lambda \in \Lambda} \mu_k^{2\lambda}
< \infty.
\end{equation}
Indeed, \eqref{pointwise condition} allows us to show that even when $\Lambda=\{\lambda_n\}_{n\geq 0}\subset \mathbb{N}$,  \eqref{MS condition-2} does not suffice to prove that  $\{D^{\lambda}b:\lambda\in\Lambda\}$ is a frame for $\ell^{2}$.
First, we need an estimate.

\begin{lem}\label{lem:xsx}
For $x>0$ let
\[
S(x)= \sum_{n=2}^\infty \exp(-nx\log(n)) .
\]
Then  $\lim_{x \to 0+} xS(x) =0$.
\end{lem}
\begin{proof}
Fix $x \in (0,1/4)$ small enough so that $N=\lfloor 1/(x \log (1/x)) \rfloor$ satisfies $N\geq 3$ (so that $\log(N-1)>0$).
Then
\[
S(x)= \sum_{n=2}^N \exp(-nx\log(n)) + \sum_{n=N+1}^\infty \exp(-nx\log(n))= S_1(x)+S_2(x), \quad \hbox{say}.
\]
For $2\le n\le N$, we have $\exp(-nx\log(n))\le 1$, so
$S_1 (x)\le N\le 1/(x \log (1/x))$. For the tail of the series,
\[
S_2(x) \le \int_{N-1}^\infty\exp(-xt\log t) \, dt.
\]
We take $u=xt \log t$ so that
\[
du=x(\log t+1)\,  dt \ge x \log t \, dt \ge x \log  (N-1) \, dt
\]
and we have as an upper bound
\[
S_2(x) \le \frac{1}{x \log (N - 1)}\int_{0}^\infty \exp(-u) \ du,
\]
so that $S_2(x) \le 1/(x \log (N-1)) \le C/(x \log (1/x))$ for some $C>0$ independent of $x$.

Putting these together, we see that $xS(x)$ is bounded by a constant times $1/\log(1/x)$ as $x \to 0+$, and hence
it tends to $0$.
\end{proof}

\begin{exm}\label{exm:nlogn}
If we now take $\lambda_n=\lceil n \log n \rceil$ or $\lambda_n=p_n$, the $n$th prime, for $n \ge 2$, both of which are greater than
$n \log n$ (in the case of $p_n$ this is Rosser's theorem~\cite{rosser}), then
each sequence satisfies the M\"untz--Szász  condition $\sum_{n=2}^\infty 1/\lambda_n = \infty$ and, by Lemma~\ref{lem:xsx}, the condition
\[
\lim_{x \to 0+} x\sum_{n=2}^\infty \exp(-\lambda_n x)=0 .
\]
Now if $\mu_k \to 1-$, then
\[
\inf_k
(1-\mu_k^2)
\sum_{n=2}^\infty
\mu_k^{2\lambda_n}
=0,
\]
since, with $x=1-\mu_k^2$ small, we have
$\mu_k^2 = 1-x \le \exp(-x)$.

Then testing the frame inequality on $e_k$ (defined in \eqref{ek}) gives
\[
\sum_{n\geq 2} |\langle e_k, t^{\lambda_n} \rangle|^2
=
(1-\mu_k^2)\sum_{n\geq 2} \mu_k^{2\lambda_n}
\to 0,
\]
while
\[
\|e_k\|_{L^2(\nu)} = 1.
\]
So no positive lower frame bound can hold.
\end{exm}

\begin{revised}
Lemma~\ref{lem:xsx} is a particular instance of the following general fact, which shows that the obstruction in Example~\ref{exm:nlogn} is that the sequences there have density zero. For $\Lambda\subset(0,\infty)$ we write
\[
N_\Lambda(r)=\#\{\lambda\in\Lambda:\ \lambda\le r\}\qquad (r>0)
\]
for its counting function.

\begin{lem}\label{lem:density}
Let $\Lambda\subset(0,\infty)$ be such that $N_\Lambda(r)<\infty$ for every $r>0$ and $N_\Lambda(r)=o(r)$ as $r\to\infty$. Then
\[
\lim_{x\to 0+} x\sum_{\lambda\in\Lambda} e^{-x\lambda}=0 .
\]
In particular, this holds whenever $\sum_{\lambda\in\Lambda}1/\lambda<\infty$.
\end{lem}

\begin{proof}
Write $N=N_\Lambda$. Since $N(r)=o(r)$ and $N$ is non-decreasing, there is $C>0$ with $N(r)\le C(1+r)$ for all $r>0$. Integrating by parts (the boundary terms vanish),
\[
\sum_{\lambda\in\Lambda}e^{-x\lambda}=\int_{(0,\infty)}e^{-xr}\,dN(r)=x\int_0^\infty N(r)e^{-xr}\,dr ,
\]
and the change of variables $r=u/x$ gives
\[
x\sum_{\lambda\in\Lambda}e^{-x\lambda}=\int_0^\infty x\,N(u/x)\,e^{-u}\,du .
\]
For each fixed $u>0$ we have $x\,N(u/x)=u\cdot\frac{N(u/x)}{u/x}\to 0$ as $x\to 0+$, and
$$x\,N(u/x)\le C(x+u)\le C(1+u)$$ for $0<x\le 1.$ By the Dominated Convergence Theorem, the integral tends to $0$.

\smallskip

Finally, if $\sum_{\lambda\in\Lambda}1/\lambda<\infty$, then $N(r)<\infty$ for all $r$ and, for $0<R<r$,
\[
\frac{N(r)}{r}\le \frac{N(R)}{r}+\sum_{\lambda\in\Lambda,\ \lambda>R}\frac1\lambda ,
\]
so that $\limsup_{r\to\infty}N(r)/r\le\sum_{\lambda>R}1/\lambda$, which tends to $0$ as $R\to\infty$.
\end{proof}

Note that the sequences in Example~\ref{exm:nlogn} satisfy $N_\Lambda(r)=O(r/\log r)$, so that Lemma~\ref{lem:density} also applies to them.
\end{revised}

As a byproduct of our previous discussion, we provide a necessary condition for $\{t^\lambda\}_{\lambda \in \Lambda}$ to be a frame for $L^2(\nu)$.

\begin{revised}
\begin{cor}
Let $\{\mu_k\}_{k\geq 0}\subset(0,1)$ be an interpolating sequence for $H^2(\mathbb{D})$, and let $\nu$ be the measure on $[0,1)$
\[
\nu=\sum_{k\ge0}(1-\mu_k^2)\delta_{\mu_k}.
\]
Let $\Lambda=\{\lambda_n\}_{n\ge0}\subset (0,+\infty)$ and assume the family $\{t^{\lambda_n}\}_{n\geq 0}$ is a frame for $L^2(\nu)$. Then the \emph{Carleson-type condition}
\begin{equation}\label{CC}
0<
\inf_{k\ge0}
(1-\mu_k^2)\sum_{n\ge0}\mu_k^{2\lambda_n}
\le
\sup_{k\ge0}
(1-\mu_k^2)\sum_{n\ge0}\mu_k^{2\lambda_n}
<\infty,
\end{equation}
holds. In particular, \emph{the Müntz--Szász condition}
\begin{equation}\label{MS}
\sum_{n\ge0}\frac1{\lambda_n}=\infty,
\end{equation}
follows \rev{and, more precisely, $\Lambda$ has positive upper density: $\displaystyle\limsup_{r\to\infty}\frac{N_\Lambda(r)}{r}>0$.}
\end{cor}
\end{revised}

\begin{proof}
In order to show \eqref{CC}, let \(A,B>0\) be frame bounds for \(\{t^{\lambda_n}\}_{n\ge0}\) in $L^2(\nu)$. Arguing as in \eqref{pointwise condition} for
$e_k=\frac1{\sqrt{w_k}}\mathbf 1_{\{\mu_k\}}$, $k\geq 0$, it follows
\[
0<
\inf_{k\ge0}
w_k\sum_{n\ge0}\mu_k^{2\lambda_n}
\le
\sup_{k\ge0}
w_k\sum_{n\ge0}\mu_k^{2\lambda_n}
<\infty,
\]
as wished.

\begin{revised}
Next, let us show that \eqref{CC} implies \eqref{MS}. If $\Lambda$ had a finite accumulation point, then
$$\sum_n\mu_k^{2\lambda_n}=\infty$$ for every $k$,
contradicting the upper bound in \rev{\eqref{CC}}; hence $N_\Lambda(r)<\infty$ for every $r>0$.

\smallskip

Suppose, for the moment, that $N_\Lambda(r)=o(r)$. Writing $x_k=1-\mu_k^2$, we have $\mu_k^2=1-x_k\le e^{-x_k}$, so that
\[
w_k\sum_{n\geq 0}\mu_k^{2\lambda_n}\le x_k\sum_{n\geq 0}e^{-x_k\lambda_n}\longrightarrow 0 \qquad (k\to\infty)
\]
by Lemma~\ref{lem:density}, since $x_k\to 0$ (recall that $\mu_k\to 1$, as $(\mu_k)$ is an infinite interpolating sequence in $(0,1)$). This contradicts the lower bound in \eqref{CC}. Therefore
$$\limsup_{r\to\infty}N_\Lambda(r)/r>0$$
and, by the last part of Lemma~\ref{lem:density}, $\sum_n 1/\lambda_n=\infty$.
\end{revised}
\end{proof}

\begin{rem}
Let us denote \(w_k=1-\mu_k^2\) for \(k\ge0\), so
$\nu=\sum_{k\ge0}w_k\delta_{\mu_k},
$
and the space \(L^2(\nu)\) is naturally identified with \(\ell^2\) by the unitary map
$$\begin{array}{lccc}
U:&L^2(\nu)&\to&\ell^2,\\
& f& \to &  Uf=\{f_k\}_{k\ge0}
\end{array}
$$
where
\begin{equation}\label{def U}
f_k=\sqrt{w_k}\,f(\mu_k), \quad (k\ge0).
\end{equation}
Indeed,
\[
\|Uf\|_{\ell^2}^2
=
\sum_{k\ge0}|f_k|^2
=
\sum_{k\ge0}|f(\mu_k)|^2w_k
=
\|f\|_{L^2(\nu)}^2.
\]

\noindent For \(\lambda_n\in\Lambda\), we have
\[
\langle f,t^{\lambda_n}\rangle_{L^2(\nu)}
=
\sum_{k\ge0}f(\mu_k)\mu_k^{\lambda_n}w_k
=
\sum_{k\ge0}f_k \sqrt{w_k} \mu_k^{\lambda_n}.
\]
Hence the frame problem in \(L^2(\nu)\) for \(\{t^{\lambda_n}\}_{n\ge0}\) is equivalent to the frame problem in \(\ell^2\) for the vectors
\begin{equation}\label{frame problem}
\beta_{\lambda_n}
=
\left(\sqrt{w_k}\mu_k^{\lambda_n}\right)_{k\ge0}, \qquad  (n\geq 0).
\end{equation}
Indeed,
\begin{equation}\label{frame problem--}
\langle f,t^{\lambda_n}\rangle_{L^2(\nu)}
=
\langle Uf,\beta_{\lambda_n}\rangle_{\ell^2}.
\end{equation}
\end{rem}

\section{Iterated multiplication and weighted composition operators}\label{sec:4}

Take $u \in H^2(\DD)$ and $\phi \in H^\infty$.
Cheshmavar~\cite{chesh} posed the question when the sequence $(\phi^n u)_{n\geq 0}$
forms a frame in $H^2(\DD)$, but did not give a complete answer. We shall solve this problem using weighted composition operators.

\smallskip

It was assumed that $\|\phi\|_\infty \le 1$, and indeed this follows from the uniform boundedness theorem,
since if we have a frame then $(\langle \phi^n u, v \rangle)$ is a bounded sequence for every
$v \in H^2$ and hence $(\phi^n u)$ is bounded in norm. \rev{Indeed,
$$\|\phi^n u\|_{H^2}^2=\frac1{2\pi}\int_0^{2\pi}|\phi(e^{i\theta})|^{2n}|u(e^{i\theta})|^2\,d\theta$$ and
$u\neq 0$ a.e. on the unit circle, this forces $|\phi|\le 1$ a.e. on the circle, that is, $\|\phi\|_\infty\le 1$.}

\smallskip

Now we use the condition given in \eqref{eq:bddsurj}, and this says that the
mapping $\rev{C^*_{(\phi^n u)}}: \ell^2 \to \rev{H^2}$ should be bounded and surjective. Writing this
as a  mapping     $H^2 \to H^2$ it says that the operator given by
\[
W: \sum_{n=0}^\infty a_n z^n \mapsto \sum_{n=0}^\infty a_n \phi^n u
\]
should be bounded and surjective. But $W$ is simply the weighted composition operator $W_{\phi,u}:=M_u C_\phi$.

\smallskip

The following result is given by Gunatillake \cite[Thm.~2.0.1]{gun}, written here in the
notation used above. Note that $W_{\phi,u}$ is injective (except in   trivial cases\rev{: if $u\neq 0$ and $\phi$ is non-constant, then $u\,(f\circ\phi)=0$ forces $f=0$ on the open set $\phi(\DD)$, hence $f=0$})
and so boundedness and surjectivity together are equivalent to invertibility.

\begin{thm}
The operator $W_{\varphi, u}$ on $H^2(\DD)$, defined by
\[
   (W_{\varphi, u} f)(z) =  u(z) f(\varphi(z)),
\]
 is invertible on $H^2(\DD)$ if and only if $\phi$ is an automorphism of $\DD$
and $u$ is both bounded and bounded away from zero on $\DD$.
\end{thm}

In view of the discussion above we have the following result.

\begin{cor}
For $\varphi \in H^\infty(\DD)$ and $u \in H^2(\DD)$ the sequence
$(\phi^n u)_{n\geq 0}$
forms a frame in $H^2(\DD)$ if and only if
$\phi$ is an automorphism of $\DD$
and $u$ is both bounded and bounded away from zero on $\DD$.
\end{cor}

\rev{Since the synthesis operator $W_{\phi,u}$ is injective, every frame of the form $(\phi^n u)_{n\geq 0}$ is in fact a Riesz basis of $H^2(\DD)$; that is, such frames are never redundant.}

Write $\mathbf{K}_w(z)=\frac{1}{1-\overline{w}z}$ for the reproducing kernel in $H^2(\DD)$, so
$f(w)=\langle f,\mathbf{K}_w\rangle$.
It is well known that for
 every $w\in\mathbb D$,
\[
W_{\varphi, u}^*\,\mathbf{K}_w \;=\; \overline{ u(w)}\,\mathbf{K}_{\varphi(w)}.
\]
With Parseval frames in mind, let us recall the following theorem by Bourdon and Narayan \cite{BourdonNarayan2010}.

\begin{thm}\label{thm:BN}
The operator $W_{\varphi, u}$ on $H^2(\DD)$, defined by
\[
   (W_{\varphi, u} f)(z) =  u(z) f(\varphi(z)),
\]
is unitary if and only if $\varphi$ is a disk automorphism and
\[
    u(z) = c \, \frac{\mathbf{K}_p(z)}{\|\mathbf{K}_p\|},
\]
where $\varphi(p) = 0$, $\mathbf{K}_p$ is the reproducing kernel at $p$, and $|c| = 1$.
\end{thm}

It is of independent interest to consider the question when
a weighted composition operator is a co-isometry: that is, its
adjoint  is an isometry (not necessarily unitary). One obvious necessary condition is
that it is isometric on reproducing kernels $\mathbf{K}_w$ for $w \in \DD$\rev{. Since}
\[
W_{\phi, u}^* \mathbf{K}_w = \overline{ u(w)} \mathbf{K}_{\phi(w)}
\]
\rev{(see above),} a necessary condition is that
\begin{equation}\label{eq:psiw}
| u(w)|^2 = \frac{1-|\phi(w)|^2}{1-|w|^2} \qquad (w \in \DD).
\end{equation}
Since operators with isometric adjoints are surjective (so in this case, invertible) we
must also have $\phi$ an automorphism and $ u$ both bounded and bounded away from zero.

Now if $\phi$ is an automorphism with a zero at $p$, then by an easy calculation this condition \eqref{eq:psiw} becomes
\[
| u(w)|^2 = \frac{1-|p|^2}{|1-\overline p w|^2} \qquad (w \in \DD),
\]
which, since $ u$ is analytic, implies by Theorem~\ref{thm:BN} that $W$ and $W^*$ are
unitary.

This can also be expressed by saying that we have a reproducing kernel hypothesis for
the isometry condition of an adjoint of a weighted composition operator.

We therefore have the following result.

\begin{thm}\rev{\label{thm:coisom}}
A weighted composition operator on $H^2(\DD)$ is a co-isometry if and only
if it is unitary.
\end{thm}

\begin{revised}

Since a frame is Parseval if and only if its synthesis operator is a co-isometry, Theorems~\ref{thm:BN} and~\ref{thm:coisom} yield the following.

\begin{cor}
For $\varphi \in H^\infty(\DD)$ and $u \in H^2(\DD)$, the sequence $(\phi^n u)_{n\geq 0}$ is a Parseval frame for $H^2(\DD)$ if and only if $\phi$ is an automorphism of $\DD$ and $u=c\,\mathbf{K}_p/\|\mathbf{K}_p\|$, where $\phi(p)=0$ and $|c|=1$. In this case $(\phi^n u)_{n\geq 0}$ is an orthonormal basis of $H^2(\DD)$.
\end{cor}
\end{revised}

\section{Model spaces and their generalizations}\label{sec:6}

Let $\theta \in H^\infty$ be an inner function. The associated model space is
defined by
\[
K_\theta= H^2 \ominus \theta H^2,
\]
the orthogonal complement of the shift-invariant subspace $\theta H^2$ in the
Hardy space $H^2$. By Beurling's theorem \cite{Beu}, every nontrivial closed subspace of $H^2$ invariant under
the unilateral shift has the form $\theta H^2$ for some inner function $\theta$.
Consequently, the spaces $K_\theta$ are invariant under the backward shift operator
\[
S^*f(z)=\frac{f(z)-f(0)}{z},
\]
and the nontrivial invariant subspaces of $S^*$ are precisely the model spaces
$K_\theta$. A good reference for model spaces is the book~\cite{GMR}.

For $w \in \DD$ we shall write $k^\theta_w \in K_\theta$ for the reproducing kernel
for evaluation at $w$. That is $\langle f, k^\theta_w \rangle = f(w)$ for all $f \in K_\theta$.

\subsection{The case of a single operator}

The paper \cite{CHP20}
contains the following observations. Note that saying $(A_1,u_1)$ is similar to $(A_2,u_2)$ means that there is an invertible $V$
such that
$A_2=V A_1 V^{-1}$ and $u_2=V u_1$.

\begin{itemize}
\item The system $(A^n u)_{n =0}^\infty $ forms a frame if and only if $(A,u)$ is similar to
either\\ (i) the compressed shift
$(S_\theta, k_0^\theta)$ acting on the model space $K_\theta$,
with $k_0^\theta$ the reproducing kernel at $0$, that is, $k_0^\theta(z)=1-\ol{\theta(0)}\theta(z)$,
or\\ (ii) the standard shift   $(S, e_0)$ acting on $H^2(\DD)$, with $e_0(z)=1$.
\item In the first case, $S_\theta$ is a completely non-unitary $C_0$ contraction, and
$\theta$ is its minimal function (unique to within a unimodular constant). That is,
$\theta$ is an inner function and
$\theta(S_\theta)=0$.
\end{itemize}

Note that we have an identity for the spectrum $\sigma(A)=\sigma(S_\theta)=\sigma(\theta)$
in case (i), and $\sigma(A)=\ol\DD$ in case (ii), although this does not
determine $\theta$ uniquely in case (i), unlike the minimal function.

\subsection{Particular cases}\label{subsec:pc}

\begin{enumerate}
\item For the frame $(A,u)$ with $A$ a diagonal operator,   discussed in Section~\ref{sec:2}, we either have $A$ similar to
a compressed shift $S_\theta$ or to
the shift $S$.
In the first case the operator is $C_0$ and its eigenvalues must form a Blaschke sequence
(clearly with no repeats). Further,
if the normalized eigenvectors are to form a Riesz basis then the eigenvalues are
a Carleson sequence.
The second case is impossible since $S$ has no eigenvalues.

\item\label{item:no2} The case when $A$ is similar to a direct sum of Jordan blocks
where the generalized eigenvectors form a Riesz basis
 corresponds to an inner function $\theta$ with repeated roots.

Observations:
\begin{itemize}
\item We can only have one block for each  eigenvalue, since $f \in \ker(S_\theta-\lambda I)$
gives $zf=\theta h + \lambda f$ for some $h \in H^2$ and $f=\theta h/(z-\lambda)$.
If this is in $K_\theta$ we have $\theta(\lambda)=0$ and, dividing by $\theta$, $h/(z-\lambda) \in \ol{H^2_0}$ so $h$ is a constant. That is, the eigenspace is one-dimensional, being spanned by $\theta/(z-\lambda)$.

It can be seen by considering the Taylor series
that the eigen\-spaces of $S_\theta^*$ are also one-dimensional, the eigenvectors being reproducing kernels, and the eigenvalues are $\ol\lambda$, where $\theta(\lambda)=0$.

\item Also, $S_\theta$ is a $C_0$ operator, with minimal function $\theta$, which, since
each zero of
$\theta$ has finite order, means that the blocks are all finite.

\begin{revised}
\item Some care is needed with the meaning of ``direct sum of Jordan blocks''. If $A$ is similar to an orthogonal direct sum $\bigoplus_k(\lambda_k I+N_k)$, where each $N_k$ is a standard nilpotent Jordan block (with ones on the superdiagonal with respect to an orthonormal basis), then only finitely many blocks can be nontrivial. Indeed, the eigenvalues are distinct zeros of $\theta$, so infinitely many nontrivial blocks would give $|\lambda_k|\to 1$ along these blocks, while for a $2\times 2$ block
\[
\|(\lambda I+N)^n\|\ \ge\ n|\lambda|^{n-1},
\]
and $\sup_n n|\lambda|^{n-1}\ge c/(1-|\lambda|)$ for some absolute constant $c>0$. Thus $A$ would not be power bounded, whereas, by \cite{CMPP20}, an operator generating a frame of iterations is similar to a contraction.

\item If, instead, one only requires the generalized eigenvectors to form a Riesz basis (with no normalization of the Jordan chains), then infinitely many nontrivial blocks may occur\rev{; see Proposition~\ref{prop:jordan} at the end of this subsection.}
\end{revised}

\end{itemize}

\item
For the multiplication frame $(M_\phi, f)$ discussed in Section~\rev{\ref{sec:4}},
we see that with $\phi$ an automorphism we are in situation (ii) above (similarity with the shift); the more complicated situation (i) does not occur.\\

\item As discussed in \cite{GPR-Hardy}, on writing $\mathcal{H}_1$ for the Hardy operator
on $L^2(0,1)$ defined by
$$(\mathcal{H}_1 f)(x) = \frac{1}{x} \int_{0}^{x} f(t) dt, \quad 0 < x < 1,$$
the fact that $I-\mathcal{H}_1^*$ is unitarily equivalent to the
unilateral shift $S$ on $H^2(\DD)$ implies that it can generate frames. This is situation (ii) once more.
\end{enumerate}

\smallskip

\begin{revised}
\noindent In what follows, for $\lambda\in\DD$ we write $\mathbf K^{(1)}_\lambda(z)=z/(1-\ol\lambda z)^2$, so that $\langle f,\mathbf K^{(1)}_\lambda\rangle=f'(\lambda)$ for $f\in H^2(\DD)$; moreover $b_\lambda(z)=(\lambda-z)/(1-\ol\lambda z)$ denotes the Blaschke factor with zero at $\lambda$.

\smallskip

\begin{prop}\label{prop:jordan}
Let $B$ be the Blaschke product with zeros $(\lambda_k)_{k\geq 0}$, a Carleson sequence in $\DD$, and let $\theta=B^2$. Then:
\begin{enumerate}
\item[(a)] The operator
\[
Tf=\Big((1-|\lambda_k|^2)^{1/2}f(\lambda_k),\ (1-|\lambda_k|^2)^{3/2}f'(\lambda_k)\Big)_{k\geq 0}
\]
maps $H^2(\DD)$ boundedly onto $\ell^2(\NN_0;\CC^2)$, and $\ker T=\theta H^2(\DD)$. Hence $T|_{K_\theta}$ is an isomorphism of $K_\theta$ onto $\ell^2(\NN_0;\CC^2)$.

\smallskip

\item[(b)] The vectors
\[
x_{k,1}=(1-|\lambda_k|^2)^{1/2}\mathbf K_{\lambda_k},\qquad x_{k,2}=(1-|\lambda_k|^2)^{3/2}\mathbf K^{(1)}_{\lambda_k}\qquad(k\geq 0)
\]
form a Riesz basis of $K_\theta$, and so does its biorthogonal system $(y_{k,1},y_{k,2})_{k\geq 0}$.

\smallskip

\item[(c)] For every $k$, $y_{k,1},y_{k,2}\in\ker(S_\theta-\lambda_kI)^2$ and $y_{k,2}\in\ker(S_\theta-\lambda_k I)$. In particular, $(y_{k,i})$ is a Riesz basis of $K_\theta$ consisting of generalized eigenvectors of $S_\theta$, and $S_\theta$ has infinitely many nontrivial $2\times 2$ Jordan blocks.
\end{enumerate}
Consequently, since $(S_\theta^nk_0^\theta)_{n\geq 0}$ is a \textup{(}Parseval\textup{)} frame for $K_\theta$, there are frames of iterations generated by operators with a Riesz basis of generalized eigenvectors and infinitely many nontrivial Jordan blocks.
\end{prop}

\smallskip

\begin{proof}
Write $\delta=\inf_k|B_k(\lambda_k)|>0$, where $B_k=B/b_{\lambda_k}$, and note that $|b_\lambda'(\lambda)|=1/(1-|\lambda|^2)$, so that
\begin{equation}\label{eq:Bprime}
|B'(\lambda_k)|=|B_k(\lambda_k)|\,|b_{\lambda_k}'(\lambda_k)|\geq \frac{\delta}{1-|\lambda_k|^2}\qquad(k\ge 0).
\end{equation}

\smallskip

Let us prove statement $(a)$.

\noindent \emph{Boundedness.} The first component of $T$ is bounded by Theorem~\ref{thm:mcphail} (or by Carleson's embedding theorem, since
$$\sum_k(1-|\lambda_k|^2)\delta_{\lambda_k}$$
is a Carleson measure \cite[Chap.~VII]{Garnett}).

For the second one, since $(\lambda_k)$ is uniformly separated, there is $c\in(0,1/2)$ such that the discs $\Delta_k=\{z:|z-\lambda_k|<c(1-|\lambda_k|)\}$ are pairwise disjoint. Standard arguments involving the average of the Cauchy formula for $f'(\lambda_k)$ over circles of radius $\rho\in(r_k/2,r_k)$, with $r_k=c(1-|\lambda_k|)$, and Cauchy--Schwarz inequality yield
 \[
|f'(\lambda_k)|^2\le \frac{C}{r_k^{4}}\int_{\Delta_k}|f|^2\,dA,
\qquad\text{so}\qquad
(1-|\lambda_k|^2)^3|f'(\lambda_k)|^2\le C'\int_{\Delta_k}|f|^2\,\frac{dA}{1-|\lambda_k|}.
\]
Summing over $k$ gives
$$\sum_k(1-|\lambda_k|^2)^3|f'(\lambda_k)|^2\le C'\int_\DD|f|^2\,d\tau,$$
where $d\tau=\sum_k (1-|\lambda_k|)^{-1}\mathbf 1_{\Delta_k}\,dA$.

Since $\tau(\Delta_k)\asymp 1-|\lambda_k|$, the discs $\Delta_k$ have radii comparable to $1-|\lambda_k|$ and $\sum_k(1-|\lambda_k|)\delta_{\lambda_k}$ is a Carleson measure, $\tau$ is a Carleson measure as well, and Carleson's embedding theorem yields
$$\int_\DD|f|^2\,d\tau\le C''\|f\|^2_{H^2}.$$
Accordingly, $T$ is a bounded operator.

\smallskip

\noindent \emph{Surjectivity.} Let $(\alpha_k),(\beta_k)\in\ell^2$. By Theorem~\ref{thm:mcphail} there is $f_1\in H^2$ with $(1-|\lambda_k|^2)^{1/2}f_1(\lambda_k)=\alpha_k$ for all $k$. Put
\[
d_k=(1-|\lambda_k|^2)^{-3/2}\beta_k-f_1'(\lambda_k),
\]
so that $\big((1-|\lambda_k|^2)^{3/2}d_k\big)\in\ell^2$ by the boundedness just proved. By \eqref{eq:Bprime},
\[
(1-|\lambda_k|^2)^{1/2}\Big|\frac{d_k}{B'(\lambda_k)}\Big|\le \delta^{-1}(1-|\lambda_k|^2)^{3/2}|d_k|,
\]
which is square-summable, so Theorem~\ref{thm:mcphail} provides $h\in H^2$ with $h(\lambda_k)=d_k/B'(\lambda_k)$ for all $k$.

Then $f=f_1+Bh\in H^2$ satisfies
$$f(\lambda_k)=f_1(\lambda_k)$$
and
$$f'(\lambda_k)=f_1'(\lambda_k)+B'(\lambda_k)h(\lambda_k)=(1-|\lambda_k|^2)^{-3/2}\beta_k$$
that is, $Tf=((\alpha_k,\beta_k))_k$. Consequently, $T$ is surjective.

\smallskip

\noindent \emph{Kernel.} $Tf=0$ means that $f$ vanishes to order at least two at every $\lambda_k$; by the inner-outer factorization, $B^2$ divides the inner factor of $f$, i.e., $f\in\theta H^2$, and the converse is clear.

\smallskip

Consequently,  $T|_{K_\theta}$ is injective with $T(K_\theta)=T(H^2)=\ell^2(\NN_0;\CC^2)$, and it is an isomorphism by the Open Mapping Theorem, which completes the proof of (a).

\smallskip

Let us now prove the statement of condition $(b)$.

\smallskip

Let $X=(T|_{K_\theta})^*:\ell^2(\NN_0;\CC^2)\to K_\theta$, an isomorphism. If $(e_{k,1},e_{k,2})$ denotes the standard orthonormal basis, then $$Xe_{k,i}=P_{K_\theta}x_{k,i}=x_{k,i},$$
because $\mathbf K_{\lambda_k}$ and $\mathbf K^{(1)}_{\lambda_k}$ are orthogonal to $\theta H^2$ (every function in $\theta H^2$ vanishes to order two at $\lambda_k$).

Thus $(x_{k,i})$, the image of an orthonormal basis under an isomorphism, is a Riesz basis of $K_\theta$, and its biorthogonal system $y_{k,i}=(X^{-1})^*e_{k,i}$ is also a Riesz basis, as claimed.

\smallskip

Finally, let us prove $(c)$.  By biorthogonality, $y_{k,i}\perp x_{j,1},x_{j,2}$ for $j\neq k$, i.e., $y_{k,i}$ vanishes to order two at every $\lambda_j$, $j\neq k$. Hence $y_{k,i}=B_k^2g$ for some $g\in H^2$, and
$$(z-\lambda_k)^2y_{k,i}=\theta\,g\,(z-\lambda_k)^2/b_{\lambda_k}^2\in\theta H^2,$$
since $(z-\lambda_k)/b_{\lambda_k}\in H^\infty$. As $\theta H^2$ is invariant under $S$, we have
$$p(S_\theta)=P_{K_\theta}p(S)|_{K_\theta}$$
for every polynomial $p$, and therefore
$$(S_\theta-\lambda_kI)^2y_{k,i}=P_{K_\theta}\big((z-\lambda_k)^2y_{k,i}\big)=0.$$

Moreover, $y_{k,2}\perp x_{k,1}$ gives $y_{k,2}(\lambda_k)=0$, so $y_{k,2}=B_k^2b_{\lambda_k}g'$ for some $g'\in H^2$, and the same argument shows $(S_\theta-\lambda_kI)y_{k,2}=0$.

 Since the eigenspaces of $S_\theta$ are one-dimensional (see 
 Section~\ref{subsec:pc} (\ref{item:no2})), $y_{k,1}$ is not an eigenvector, and each generalized eigenspace $\ker(S_\theta-\lambda_kI)^2=\spam\{y_{k,1},y_{k,2}\}$ carries a nontrivial $2\times2$ Jordan block, which yields (c) and concludes the proof.
\end{proof}

\medskip

\noindent With respect to the basis $(y_{k,2},y_{k,1})$ of $\ker(S_\theta-\lambda_kI)^2$, the restriction of $S_\theta$ has matrix $\begin{pmatrix}\lambda_k& c_k\\ 0&\lambda_k\end{pmatrix}$ with $c_k\neq 0$. Since $(y_{k,i})$ is a Riesz basis and $S_\theta$ is a contraction, the powers of these $2\times 2$ matrices are uniformly bounded, so $n|c_k||\lambda_k|^{n-1}\le C$ for all $n,k$, which forces $|c_k|\le C'(1-|\lambda_k|)$. This is consistent with the previous observation: the similarities taking these blocks to standard Jordan blocks are not uniformly bounded.
\end{revised}

\subsection{Parseval frames}

Starting with $(S, e_0)$ acting on $H^2(\DD)$ we have $\langle f, S^n e_0 \rangle = \hat f(n)$,
the $n$th Taylor coefficient,
so this is clearly a Parseval frame.

Moreover, for  $(S_\theta, k_0^\theta)$ acting on $K_\theta$ we have
\[
\langle f, S_\theta^n k_0^\theta \rangle = \langle (S^*)^n f, k_0^\theta\rangle = \hat f(n)
\]
once more, since $k_0^\theta$ is the reproducing kernel at $0$  for $K_\theta$.
Thus we have

\begin{cor}\label{cor final}
Every frame of the form $(A^n u)_n$ is similar to a Parseval frame.
\end{cor}

\smallskip
\rev{Note that the statement of Corollary \ref{cor final} could be addressed directly: if $S$ is the frame operator of $(A^nu)_n$, then $(S^{-1/2}A^nu)_n$ is the canonical Parseval frame, and $S^{-1/2}A^nu=B^n v$ with $B=S^{-1/2}AS^{1/2}$ and $v=S^{-1/2}u$. The model space approach, however, identifies the Parseval frame concretely.}

\smallskip

A generalization of the model space approach described above is due to
Cabrelli, Molter and Su\'arez~\cite{CMS23}.
They give an answer to the following question:

{\em
Let $A$ be a bounded operator on a Hilbert space $H$ and $\{u_1,\ldots,u_m\}$ an
independent  set in  $H$. When is
$\{A^n u_k: n \ge 0, k=1,\ldots,m\}$   a frame?}

The answer can be expressed in terms of invariant subspaces of the iterated shift $S^m$
on $H^2$, and these in turn can be linked to the invariant subspaces of
the shift on the vector-valued Hardy space $H^2(\DD,\CC^m)$ by using the
 following commutative diagram:

\begin{align*}\begin{CD}
H^2 @> S^m  >>H^2\\
@V J   VV @ VV J \;V\\
H^2(\DD, \CC^m) @> S >> H^2(\DD, \CC^m)
\end{CD},  \end{align*}
where $J: H^2 \to H^2(\DD, \CC^m)$ is the unitary mapping
\[
J: \sum_{k=1}^m z^{k-1}f_k(z^m) \mapsto (f_1(z),f_2(z),\ldots, f_m(z)).
\]
for $f_1,\ldots, f_m \in H^2$ together with the
Beurling--Lax theorem (e.g. \cite[Thm.~3.1.7]{LOLS}).
Note that a function $\Phi \in H^\infty(\DD, \LL(\CC^r,\CC^m))$ is said to be {\em inner\/} if
$\Phi(e^{i\omega})$ is an isometry for almost all real $\omega$.

\begin{thm} (Beurling--Lax) \label{thm:kinvsub}
Let $K$ be a non-zero closed subspace of $H^2(\DD, \CC^m)$ that is invariant under the
shift $S$. Then there is an $r$ with $1 \le r \le m$ and an inner function
$\Phi \in H^\infty(\DD,  \LL(\CC^r,\CC^m))$ such that $K=\Phi H^2(\DD, \CC^r)$.
\end{thm}

That is, the subspaces  of $H^2(\DD)$ invariant under $S^m$ have the form $J^{-1}K$, where $K$ is
given in Theorem~\ref{thm:kinvsub}.

\rev{Now let $U:H^2\to H$ be the synthesis operator of $\{A^nu_k: n\ge 0,\ k=1,\ldots,m\}$, after identifying $\ell^2(\NN_0\times\{1,\ldots,m\})$ with $H^2$ via
$$e_{(n,k)}\mapsto z^{mn+k-1},$$
that is, $U z^{mn+k-1}=A^nu_k$. Then $US^m=AU$, so that $\ker U$ is invariant under $S^m$, and $\ker U=\{0\}$ or $J\ker U=K$ for some $K$ as in Theorem~\ref{thm:kinvsub}.}

We next have that $\rang U^* = (\ker U)^\perp =: M$, a space invariant under $(S^*)^m$ and that $U:M \to H$
is an isomorphism.
Arguments similar to those in \cite{CHP20} now show that $A$  is similar to the compressed shift
$P_M(S^m_{|M})$. The simplest case is $M=H^2$, which makes $A$ similar to $S^m$.

Under the transformations given, we wish to find vectors corresponding to $u_1,\ldots,u_m$.
In the case $K=H^2(\DD,\CC^m)$
these will be $v_j= JU^{-1}u_j= e_{j}$, for $j=1,\ldots,m$, the constant function
that takes the value $(0,\ldots,0,1,0, \ldots 0)$ with the $1$ in the $j$th place.
In the more general case $K=\Phi H^2(\DD,\CC^r)$
it will be the projection of $v_j$ onto $K^\perp$, namely the vectorial reproducing kernel $k^M_{j,0}$
such that $\langle (f_1,\ldots,f_m),k^M_{j,0} \rangle = f_j(0)$ for
$(f_1,\ldots,f_m) \in \rev{JM=K^\perp}$. Once more these are Parseval frames.

\begin{revised}
\subsection{Final remark: shifts of infinite multiplicity}\label{subsec:infinite}

\noindent It is natural to ask what happens when infinitely many vectors are iterated, that is, for frames of the form $\{A^nu_j:\ n\geq 0,\ j\geq 1\}$. Identifying $\ell^2(\NN_0\times\NN)$ with the vector-valued Hardy space $H^2(\DD;\ell^2)$ via $e_{(n,j)}\mapsto z^ne_j$, the synthesis operator $U$ satisfies $U(z^ne_j)=A^nu_j$ and
\[
US_\infty=AU,
\]
where $S_\infty$ denotes the shift of infinite multiplicity, that is, multiplication by $z$ on $H^2(\DD;\ell^2)$. As before, $\ker U$ is invariant under $S_\infty$, and the Beurling--Lax--Halmos theorem \cite{Halmos} describes it as $\Theta H^2(\DD;E)$ for some closed subspace $E\subseteq\ell^2$ and some inner function $\Theta\in H^\infty(\DD,\LL(E,\ell^2))$. Then $M=(\ker U)^\perp$ is invariant under $S_\infty^*$, and $A$ is similar to the compression $P_MS_\infty|_M$, that is, $A^*$ is similar to $S_\infty^*|_M$.

\smallskip

However, in contrast with the finite multiplicity case, this situation is completely general. Indeed, if $A$ is any bounded operator on $H$ with spectral radius $r(A)<1$ and $(e_j)_{j\ge1}$ is an orthonormal basis of $H$, then
\[
\sum_{n\geq 0}\sum_{j\geq 1}|\langle x,A^ne_j\rangle|^2=\sum_{n\geq 0}\|A^{*n}x\|^2\qquad(x\in H),
\]
which lies between $\|x\|^2$ and $\big(\sum_{n\ge0}\|A^n\|^2\big)\|x\|^2<\infty$. Hence
$$\{A^ne_j:\ n\ge0,\ j\ge1\}$$
is always a frame for $H$, and its analysis operator $x\mapsto\sum_{n}z^n\otimes A^{*n}x$ is precisely the intertwining operator in Rota's model theorem \cite{Rota}: every operator with spectral radius less than one is similar to the restriction of $S_\infty^*$ to an invariant subspace. Since every bounded operator becomes of this type after multiplication by a suitable non-zero scalar, $S_\infty^*$ is a \emph{universal operator} in the sense of Rota (see also Caradus \cite{Caradus} and \cite{CP11}).

\smallskip

As a consequence, the lattice of invariant subspaces of $S_\infty^*$ encodes the invariant subspaces of every bounded operator on a separable infinite-dimensional Hilbert space; in particular, the invariant subspace problem has a positive answer if and only if every minimal non-zero invariant subspace of $S_\infty^*$ is one-dimensional (see \cite{CP11}).
\smallskip

Therefore, although the Beurling--Lax--Halmos theorem parametrizes the kernels of the synthesis operators by inner functions $\Theta$, a description of frames generated by infinitely many vectors in terms of the associated compressed shifts, analogous to the one given above for finitely many vectors, is far beyond the reach of the present approach: it would provide, among other things, a description of the invariant subspaces of arbitrary operators.
\end{revised}

\end{document}